\documentclass{amsart}[12pt]
\usepackage{mathrsfs, amsmath,amssymb}
\usepackage{fullpage}
\newtheorem{teo}{Theorem}

\newtheorem{lem}{Lemma}
\newtheorem{Rem}{Remark}
\title{Markov semi-groups associated with the complex unimodular group $Sl(2,\mathbb{C})$}
\begin{document}
\author[N. Demni]{Nizar Demni}
\address{IRMAR, Universit\'e de Rennes 1\\ Campus de
Beaulieu\\ 35042 Rennes cedex\\ France}
\email{nizar.demni@univ-rennes1.fr}
\maketitle

\begin{abstract}
In this paper, we derive the explicit expressions of the Markov semi-groups constructed by P. Biane in \cite{Bia1} from the restriction of a particular positive definite function on the complex unimodular group $SL(2,\mathbb{C})$ to two commutative subalgebras of its universal $C^{\star}$-algebra. Our computations use Euclidean Fourier analysis together with the generating function of Laguerre polynomials with index $-1$, and yield absolutely-convergent double series representations of the semi-group densities. 
We also supply some arguments supporting the coincidence, noticed by Biane as well, occurring between the heat kernel on the Heisenberg group and the semi-group corresponding to the intersection of the principal and the complementary series. To this end, we appeal to the metaplectic representation $Mp(4,\mathbb{R})$ and to the Landau operator in the complex plane. 
\end{abstract} 

\section{Reminder: Intertwining operators arising from Gelfand pairs} 
Two Markov semigroups $(P_t)_{t \geq 0}$ and $(Q_t)_{t \geq 0}$ defined on measurables spaces $(E, \mathscr{E})$ and $(F, \mathscr{F})$ respectively are said to be intertwined by a Markov kernel 
\begin{equation*}
\Lambda: (E, \mathscr{E}) \rightarrow (F, \mathscr{F})
\end{equation*}
if
\begin{equation}\label{Inter}
Q_t \Lambda = \Lambda P_t.
\end{equation}
When \eqref{Inter} holds, it allows for instance to transfer (under regularity assumptions on $\Lambda$) analytical and probabilistic properties of the semigroup $(P_t)_{t \geq 0}$ to $(Q_t)_{t \geq 0}$ as it is the case of the Brownian and the Dunkl semi-groups (see e.g. \cite{CDGRVY}, CH. II). However, proving the existence of an intertwining relation and/or constructing the Markov kernel $\Lambda$ is in general not obvious. For instance, numerous examples of intertwining operators were explicitly computed in \cite{CPY} using the filtering procedure (see also \cite{Hir-Yor} for further developments). In \cite{BBO}, it was shown that the semigroups of the Brownian motion in a finite-dimensional Euclidean space and of the Brownian motion conditioned to stay in the interior of the Weyl chamber of a finite Coxeter group are interwtwined by means of the so-called Duistermaat-Heckman measure. Though this intertwining uses the filtering procedure as well, it relies heavily on the action of the Coxeter group on the underlying Euclidean space which allows in this setting for the construction of the Brownian motion from its conditioned process through the so-called generalized Pitman transforms. In the same vein and as explained in \cite{Bia1}, the existence of a coupling between two Markov processes together with a suitable group action provides an intertwining relation between their semigroups. A noncommutative version of this construction was used in \cite{Bia} to produce a Markov kernel intertwining the so-called noncommutative Bessel semigroup and the heat semi-group on $\mathbb{R}^2$. In that version, the noncommutativity is only concerned with the coupling process which corresponds in this new picture to a completely-positive contraction semigroup $(T_t)_{t \geq 0}$ on a noncommutative $C^{\star}$-algebra $A$. As to the intertwining relation \eqref{Inter}, it holds between the restrictions of $(T_t)_{t \geq 0}$ to two commutative subalgebras $B$ and $C$ of $A$. Moreover, the substitute of the aforementioned group action is now a completely positive projection $\pi: A \rightarrow C$ which intertwines $T_t$ and its restriction to $C$ at any time $t > 0$. Further examples of intertwinings arising from these considerations were further provided and analyzed in \cite{Bia1}. One of them arises from the diagonal subgroup of the complex unimodular group $SL(2,\mathbb{C})$ and from the Gelfand pair $(SL(2, \mathbb{C}), SU(2))$, yet lacked the explicit expressions of the corresponding Markov semigroups and in turn of the intertwining kernel. More precisely, let
\begin{equation*}
T := \left\{\left(\begin{array}{lr}
e^x & 0 \\ 
0 & e^{-x} 
\end{array}\right), \quad x \in \mathbb{R}\right\},
\end{equation*}
be the diagonal subgroup of $SL(2,\mathbb{C})$. Then, by the virtue of the Cartan decomposition, a $SU(2)$-bi-invariant function on $SL(2,\mathbb{C})$ depends only on the real variable $x$. Besides, the map 
\begin{equation*}
\varphi: x \mapsto x\coth(x)-1 = x\frac{1+e^{-2x}}{1-e^{-2x}} - 1
\end{equation*}
defines a $SU(2)$-biinvariant conditionally positive definite function (\cite{BSW}) or equivalently (by Schoenberg Theorem), 
\begin{equation*}
\psi_t: x \mapsto e^{t(1-x\coth(x))} 
\end{equation*} 
is a continuous $SU(2)$-biinvariant positive definite function. The function $\varphi$ is obtained from a suitable limit of positive definite spherical functions along a path in the complementary series tending to the trivial representation. In the representation-theoretical realm, such a function belongs to the so-called Lie cone and corresponds by the GNS construction to an infinitesimal small representation in a neighborhood of the trivial one together with an unbounded cocycle (\cite{Ver-Kar}).

On the other hand, $\phi$ is the L\'evy exponent of a background driving L\'evy process generating the self-decomposable random variable whose characteristic function is $x/\sinh(x)$ (\cite{Jur-Yor}). In particular, the following 
L\'evy-Kintchine formula holds: 
\begin{equation}\label{LevKin}
1-x\coth(x) = \int_{\mathbb{R}\setminus \{0\}} (e^{iux} -1) \frac{\pi}{4} \frac{du}{\sinh^2(\pi u/2)}, 
\end{equation}
whence one deduces the action of the corresponding infinitesimal generator on sufficiently regular functions (see e.g. \cite{App}, Theorem 3.3.3., \cite{BSW}, Corollary 2.9): 
\begin{align*}
\mathscr{L}(f)(u) & = \frac{1}{2\pi} \int_{\mathbb{R}}e^{iux}[1-x\coth(x)]\mathcal{F}(f)(x) dx 
\\& = \frac{\pi}{4}  \int_{\mathbb{R}\setminus \{0\}}(f(u+v) - f(u)) \frac{dv}{\sinh^2(\pi v/2)},
\end{align*}
where $\mathcal{F}$ stands for the Euclidean Fourier transform. A non commutative approach to this L\'evy process stems from the previous considerations. Indeed, let $\psi_t$ act by multiplication on the convolution algebra $L^1(Sl(2, \mathbb{C}))$, then this action defines a completely positive contraction semigroup on the universal $C^{\star}$-algebra and leaves invariant the commutative subalgebras $L^1(T)$. By completing $L^1(T)$ to the universal commutative $C^{\star}$-algebra $C^{\star}(T)$, the L\'evy semi-group $(Q_t)_{t \geq 0}$ associated with $(\psi_t)_{t \geq 0}$ arises from the restriction of the aforementioned action to $C^{\star}(T)$. More concretely, the Gelfand spectrum of $C^{\star}(T)$ is isomorphic to $\mathbb{R}$ through the Euclidean Fourier transform and one has:
\begin{equation*}
\mathcal{F}^{} \left(f\psi_t\right)(-x) =  \int_{\mathbb{R}} \mathscr{F}(f)(u)q_t(x-u)du,
\end{equation*}
for suitable $f \in L^1(T)$.

In \cite{Bia1}, another new Markov semi-group was introduced by Biane using the fact that $(SL(2,\mathbb{C}), SU(2))$ is a Gelfand pair. In this case, the Gelfand spectrum of the convolution algebra of $SU(2)$-biinvariant functions is given by the set of bounded spherical functions: 
\begin{equation*}
x \mapsto \phi_{\omega}(x) = \frac{\sinh(\omega x)}{\omega \sinh(x)}, \quad \omega \in \Omega,
\end{equation*}
where $\Omega = i\mathbb{R} \cup [-1,1]$. In particular, for any $\omega \in \Omega$, one obtains through the Gelfand-Fourier transform a Markov semi-group $(P_t)_{t \geq 0}$ such that: 
\begin{equation}\label{Godement}
\phi_{\omega}(x) \psi_t(x) = \int_{\Omega} \phi_{\xi}(x) P_t(\omega, d\xi).
\end{equation}  
Since $\phi_{\omega}$ is also a positive definite function, then the RHS of \eqref{Godement} is nothing else but the decomposition of the positive definite function $\phi_{\omega} \psi_t$ into extreme ones (see e.g. \cite{God}). 

In this paper, we shall give explicit expressions of semi-groups of $(Q_t)_{t \geq 0}$ and $(P_t)_{t \geq 0}$. Our computations use the euclidean Fourier transform and appeal to the generating function of Laguerre polynomials with index $-1$ which are related to the so-called Lah numbers (\cite{Boy}). By standard arguments from Fourier analysis, $Q_t, t > 0,$ is absolutely continuous with respect to Lebesgue measure on $\mathbb{R}$. As noticed in \cite{Bia1}, the same holds true for $P_t, t > 0,$ when either $\omega \in i\mathbb{R}$ (the principal series) or $\omega \in  [-1,1],  t \leq |w|$ (the complementary series). Otherwise, an extra atom shows up in the Lebesgue decomposition of $P_t(w, d\xi), |w| > 0, 0 < t < |w|$. In the last part of the paper, we supply some arguments supporting the occurrence (noticed also by Biane in the same paper) of $\phi_0\psi_t$ in the subelliptic heat kernel of the Heisenberg group or equivalently in the L\'evy stochastic area formula (\cite{Gav}). To this end, we appeal to the Schr\"odinger operator with a constant magnetic field in the plane (sometimes referred to as the Euclidean Landau laplacian, \cite{AIM}) and its realization by means of the metaplectic representation of $Mp(4, \mathbb{R})$ (\cite{Mat-Uek}). In particular, the heat semi-group of the Landau Laplacian may be interpreted via the GNS construction as an average of a continuous family of unitary representations with respect to the two-dimensional Gaussian distribution. However, we still do not know whether or not these representations occur in the metaplectic representation though the answer seems to be positive. Indeed, it turns out that the radial part of the symplectic matrix representing this Schr\"odinger operator is the tensor product of a $2 \times 2$ matrix with the identity matrix. Since the maximal compact subgroup $O(4) \cap Sp(4,\mathbb{R})$ of the symplectic group may be identified with the unitary group $U(2)$, then we can turn the Cartan decomposition of this matrix in $Sp(4,\mathbb{R})$ into a Cartan decomposition of another one in $Sl(2,\mathbb{C})$. 

The paper is organized as follows. In section 2, we write down the semi-group density of $(Q_t)_{t \geq 0}$ while section 3 contains the (three) different expressions of the kernel $P_t(\omega, d\xi)$ according to the values of the spectral parameter $\omega$. The last section is devoted to the special value $\omega = 0$ in relation with the Heisenberg group and the the metaplectic representation. .

\section{Explicit expression of the L\'evy semigroup $(Q_t)_{t \geq 0}$}  
It is clear from \eqref{LevKin} that $\psi_1$ is the characteristic function of an infinitely divisible distribution on the real line generating the L\'evy semi-group $(Q_t)_{t \geq 0}$ (see Corollary 1 in \cite{Jur-Yor}). Moreover, since $\psi_t$ is integrable for any $t > 0$, then the corresponding semi-group density $q_t$ is smooth (\cite{Kno-Sch}, Theorem 2.1) and may be computed using the Fourier inversion formula: 
\begin{equation*}
q_t(\xi) = \frac{1}{2\pi} \int_{\mathbb{R}}e^{-i\xi x} \psi_t(x) dx = \frac{e^t}{2\pi} \int_{\mathbb{R}}e^{-i\xi x} e^{-tx\coth(x)} dx, \quad \xi \in \mathbb{R}.
\end{equation*}
The following theorem provides a double-series representation for $q_t(\xi)$: 
\begin{teo}\label{Th1}
For any $t > 0$ and any $u \in \mathbb{R}$, 
\begin{align*}
q_t(\xi) = \frac{e^t}{\pi} \sum_{m \geq 0}  \sum_{j \geq 0} \frac{(-2t)^m (m)_j}{j![(2j+2m+t)^2 + \xi^2]^{(m+1)/2}} T_{m+1}\left(\frac{2j+2m+t}{\sqrt{(2j+2m+t)^2+\xi^2}}\right),
\end{align*}
where $T_{m+1}$ is the $(m+1)$-th Tchebycheff polynomial of the first kind (\cite{AAR}): 
\begin{equation*}
T_{m+1}(\cos a) = \cos((m+1)a), \quad a \in \mathbb{R}. 
\end{equation*}
Moreover, the double series is absolutely convergent.  
\end{teo}
\begin{proof} 
Expand 
\begin{equation}\label{GenFun}
e^{-tx\coth(x)} = e^{-tx} \sum_{j \geq 0} L_j^{(-1)}(2tx)e^{-2jx}, x \geq 0, 
\end{equation}
where 
\begin{equation}\label{Laguerre}
L_j^{(-1)}(x) = \frac{1}{j!} \sum_{m=0}^j (-1)^m \binom{j}{m} (m)_{j-m} x^m,  \quad j \geq 0,
\end{equation}
are the Laguerre polynomials of index $-1$ (\cite{AAR}) and for a real number $a$,
\begin{equation*}
(a)_j = a(a+1)\cdots(a+j-1), 
\end{equation*}
is the Pochhammer symbol with the convention $(0)_0 = 1$, $(0)_j = 0, j \geq 1$. Note that there is no constant term in these polynomials provided that $j \geq 1$ and that they are connected to Lah numbers (\cite{Boy}). 
Now, the following bound (see e.g. \cite{Lew-Szy}, p.530):
\begin{equation}\label{bound}
|L_j^{(-1)}(2tx)| \leq \frac{C}{j^{1/4}} e^{tx},
\end{equation}
where $C$ is an absolute constant, allows to apply Fubini Theorem to compute
\begin{align*}
 \int_{\mathbb{R}}e^{-i\xi x} e^{-tx\coth(x)} dx & = 2\int_{0}^{\infty} \cos(\xi x)  \sum_{j \geq 0} L_j^{(-1)}(2x)e^{-(2j+t)x} dx 
\\& = 2\sum_{j \geq 0}  \int_{0}^{\infty}\cos(\xi x) e^{-(2j+t)x}L_j^{(-1)}(2tx) dx
\\& =  2\sum_{j \geq 0} \sum_{m=0}^j \frac{(-2t)^m}{(j-m)!m!}(m)_{j-m}  \int_{0}^{\infty} \cos(\xi x) e^{- (2j+t)x}x^{m} dx.
\end{align*}
Using formula 3.944. 6., in \cite{Gra-Ryz}, we further get: 
\begin{align*}
 \int_{\mathbb{R}}e^{-i\xi x} e^{-tx\coth(x)} dx & =  2\sum_{j \geq 0} \sum_{m=0}^j \frac{(-2t)^m}{(j-m)!}\frac{(m)_{j-m}}{[(2j+t)^2 + \xi^2]^{(m+1)/2}} \cos\left\{(m+1)\arctan\left(\frac{\xi}{2j+t}\right)\right\}
\\& = 2 \sum_{j \geq 0} \sum_{m=0}^j \frac{(-2t)^m}{(j-m)!}\frac{(m)_{j-m}}{[(2j+t)^2 + \xi^2]^{(m+1)/2}} T_{m+1}\left\{\cos\arctan\left(\frac{\xi}{2j+t}\right)\right\},
\end{align*}
which together with the identity
\begin{equation*}
\cos(y) = \frac{1}{\sqrt{1+\tan^2(y)}}, \quad y \in [-\frac{\pi}{2}, \frac{\pi}{2}], 
\end{equation*}
yield
\begin{align}\label{DS}
q_t(\xi) = \frac{e^t}{\pi} \sum_{j \geq 0}\sum_{m=0}^j \frac{(-2t)^m}{(j-m)!}\frac{(m)_{j-m}}{[(2j+t)^2 + \xi^2]^{(m+1)/2}} T_{m+1}\left(\frac{2j+t}{\sqrt{(2j+t)^2+\xi^2}}\right).
\end{align}
Finally, use the bound
\begin{equation*}
(2j+t+2m)^2 \geq 4(j+m)^2,
\end{equation*} 
together with the Gamma integral: 
\begin{equation*}
\frac{1}{(j+m)^{m+1}} = \frac{1}{m!}\int_0^{\infty} v^m e^{-(j+m)v} dv 
\end{equation*}
to see that 
\begin{align*}
\sum_{j \geq 1} \sum_{m=1}^j\frac{(2t)^m}{(j-m)!}\frac{(m)_{j-m}}{[(2j+t)^2 + \xi^2]^{(m+1)/2}} & = \sum_{m \geq 1} \sum_{j \geq m} \frac{(2t)^m}{(j-m)!}\frac{(m)_{j-m}}{[(2j+t)^2 + \xi^2]^{(m+1)/2}}
\\& = \sum_{m \geq 1} \sum_{j \geq 0} \frac{(2t)^m}{j!}\frac{(m)_{j}}{[(2j+2m+t)^2 + \xi^2]^{(m+1)/2}}
\\& \leq \sum_{m \geq 1}t^m \sum_{j \geq 0} \frac{(m)_j}{j!} \frac{1}{(j+m)^{m+1}}
\\& = \int_0^{\infty}\sum_{m \geq 1} \frac{(tv)^m}{m!(e^{v}-1)^m} dv  
\\& = \int_0^{\infty}\left(e^{tv/(e^v-1)} - 1\right) dv \quad < \infty
\end{align*}
uniformly in $\xi$. As a matter of fact, the double series in the RHS of \eqref{DS} converges absolutely and the sought expression for $q_t$ follows after inverting the summation order in \eqref{DS}.
\end{proof} 

\begin{Rem}
Recall from \eqref{LevKin} that the L\'evy measure of the semi-group $(q_t)_{t \geq 0}$ is given by: 
\begin{equation*}
\frac{\pi}{4} \frac{du}{\sinh^2(\pi u/2)} du, \quad u \neq 0. 
\end{equation*}
Recall also from Lemma 2.16 in \cite{BSW} that this measure is uniquely determined by the vague limit: 
\begin{equation*}
\lim_{t \rightarrow 0^+} \frac{1}{t}\int_{\mathbb{R} \setminus \{0\}} f(\xi)q_t(\xi) d\xi = \frac{\pi}{4} \int_{\mathbb{R} \setminus \{0\}} f(\xi) \frac{d\xi}{\sinh^2(\pi \xi /2)},
\end{equation*}
for compactly-supported functions $f$ in $\mathbb{R} \setminus \{0\}$. Since $(0)_j = \delta_{j0}$ and since 
\begin{equation*}
T_1(u) = u, \quad T_2(u) = 2u^2-1,
\end{equation*}
then quick computations show: 
\begin{equation*}
\lim_{t \rightarrow 0^+} \frac{1}{t} q_t(\xi) = \frac{1}{\pi \xi^2} - \frac{1}{2\pi}\sum_{j \geq 1}\frac{j^2 - (\xi/2)^2}{[(j^2+(\xi/2)^2]^2}. 
\end{equation*}
Consequently, the Lebesgue convergence Theorem leads to the following identity: 
\begin{equation*}
 \frac{1}{\sinh^2(\pi \xi /2)} = \frac{4}{\pi^2 \xi^2} - \frac{2}{\pi^2}\sum_{j \geq 1}\frac{j^2 - (\xi/2)^2}{[(j^2+(\xi/2)^2]^2}, 
\end{equation*}
which may be derived from 
\begin{equation*}
\pi \cot(\pi x) = \frac{1}{x} + \sum_{j \geq 1}\frac{2x}{x^2-j^2}
\end{equation*}
after complexification and differentiation. 
\end{Rem}

\begin{Rem} 
The Tchebycheff polynomial may be expanded as (\cite{AAR}): 
\begin{equation*}
T_{m+1}(y) = \sum_{k=0}^{[(m+1)/2]} a_{k,m} y^{m+1-2k},
\end{equation*}
for some real coefficients $a_{k,m}$. Plugging this expansion in \eqref{DS}, we get
\begin{align*}
q_t(\xi) = \frac{e^t}{\pi} \sum_{j \geq 0}\sum_{m=0}^j \frac{(-2t)^m(m)_{j-m}}{(j-m)!} \sum_{k=0}^{[(m+1)/2]} a_{k,m} \frac{(2j+t)^{m+1-2k}}{[(2j+t)^2 + \xi^2]^{m+1-k}}.
\end{align*}
On the other hand, the semi-group density of the Cauchy process in $\mathbb{R}^d, d \geq 1,$ starting at the origin reads (\cite{App}): 
\begin{equation*}
\mathscr{C}_{s,d}(y) = \frac{1}{\pi^{(d+1)/2}} \Gamma\left(\frac{d+1}{2}\right)\frac{s}{(s^2+|y|^2)^{(d+1)/2}}, \quad y \in \mathbb{R}^d,  s > 0. 
 \end{equation*}
Hence, $q_t$ may be seen as a superposition of radial parts of Cauchy semi-group densities with varying even dimensions. 
\end{Rem}

\begin{Rem}
Let $(R_t)_{t \geq 0}$ be a Bessel process of dimension $\delta > 0$ and starting at $R_0 = 0$. In \cite{Gho}, Theorem 2.1, the author derives an explicit expression for the density $f_a^{\delta}$ of the conditional distribution of (see also \cite{Pit-Yor}):
\begin{equation*}
\int_0^1R_s^2 ds 
\end{equation*}
given $R_1 = a$. In particular, the Laplace transform of this random integral at $x^2/2$ is given by 
\begin{equation*}
\left(\frac{x}{\sinh(x)}\right)^{\delta/2} e^{(a^2/2)(1-x\coth(x))}, 
\end{equation*}
which tends to $ \psi_{a^2/2}(x)$ as  $\delta \rightarrow 0^+$. However, 
\begin{equation*}
\lim_{\delta \rightarrow 0^+} f_a^{\delta} \neq q_{a^2/2}, 
\end{equation*} 
which may be checked directly from the double series representation of $f_a^{\delta}$. 
\end{Rem}

\section{Explicit expression of the kernel $P_t(w, d\xi), w \in \Omega$}
The Lebesgue decomposition of the kernel $P_t(w, d\xi)$ depends on the decay of $\phi_{\omega} \psi_t$ at infinity and more precisely on the square-integrability of this function with respect to the radial part of Haar measure. Since the latter is given by 
$\sinh^2(x) dx$, we distinguish the three cases according to whether $\omega$ belongs to the principal series or to the complementary one with either $t \geq |\omega|$ or $t < |\omega|, \omega \neq 0$.  
\subsection{The principal series} 
Let $\omega \in i\mathbb{R}$, then $\phi_{\omega}\psi_t$ is square integrable with respect to the Haar measure. As a matter of fact, the decomposition \eqref{Godement} is in this case only over the principal series since the complementary one does not appear in the decomposition of the left regular representation into irreducible ones (see e.g. \cite{Ber}, \cite{GGS}). In this respect, we shall prove the following: 
\begin{teo}\label{Th2}
Let $\omega \in \mathbb{R}$ and $x > 0$, then 
\begin{equation*}
\phi_{i\omega}(x) e^{t(1-x\coth(x))} = \int_{\mathbb{R}} \phi_{i\xi}(x) \frac{q_t(\xi-\omega) - q_t(\xi+\omega)}{2\omega} \xi d\xi. 
\end{equation*}
In particular, $P_t(\omega, d\xi)$ is absolutely continuous with respect to Lebesgue measure $d\xi$. 
 \end{teo}

\begin{proof}
Since $\phi_{i\omega} = \phi_{-i\omega}$, then we shall assume $w \in \mathbb{R}_+$. Then 
\begin{equation*}
\frac{\sin(\omega x)}{\omega} = \frac{x}{2\omega} \int e^{ix\xi} {\bf 1}_{[-\omega,\omega]}(\xi)d\xi, 
\end{equation*}
where for $\omega=0$, one uses the weak limit 
\begin{equation*}
\lim_{\omega \rightarrow 0^+} \frac{1}{2\omega}{\bf 1}_{[-\omega,\omega]} = \delta_0. 
\end{equation*}
It follows that 
\begin{align*}
\frac{\sin(\omega x)}{\omega} e^{t(1-x\coth(x))} & = \frac{x}{2\omega}\int e^{ix\xi} \left[q_t \star {\bf 1}_{[-\omega,\omega]}\right](\xi) d\xi,
\\& = \frac{x}{2\omega}\int_{\mathbb{R}} \cos(x\xi) \left[q_t \star {\bf 1}_{[-\omega,\omega]}\right](\xi) d\xi,
\end{align*}
where 
\begin{equation*}
\left[q_t \star {\bf 1}_{[-\omega,\omega]}\right](\xi) = \int_{\mathbb{R}} q_t(\xi-u){\bf 1}_{[-\omega,\omega]}(u) du =  \int_{-\omega}^\omega q_t(\xi-u) du.
\end{equation*}
Now, it is straightforward from the double series representation of $q_t$ (or from the Riemann-Lebesgue Lemma) that 
\begin{equation*}
\lim_{u \rightarrow \pm \infty} q_t(u) = 0, 
\end{equation*}
so that an integration by parts yields: 
\begin{align*}
\frac{\sin(\omega x)}{\omega} e^{t(1-x\coth(x))} & = -\frac{1}{2\omega}\int_{\mathbb{R}} \sin(x\xi) \partial_{\xi} \left[q_t \star {\bf 1}_{[-\omega,\omega]}\right](\xi) d\xi.
\end{align*}
But 
\begin{align*}
\partial_{\xi} \left[q_t \star {\bf 1}_{[-\omega,\omega]}\right](\xi) &= \partial_{\xi} \int_{-\omega}^{\omega} q_t(\xi - y) dy 
\\& = \int_{-\omega}^{\omega} -\partial_y(q_t(\xi - \cdot))(y) dy 
\\& = q_t(\xi+\omega) - q_t(\xi-\omega).
\end{align*}
As a result, 
\begin{align*}
\frac{\sin(\omega x)}{\omega} e^{t(1-x\coth(x))} = \int_{\mathbb{R}}  \frac{\sin(x\xi)}{\xi} \frac{q_t(\xi-\omega) - q_t(\xi+\omega)}{2\omega} \xi d\xi,
\end{align*}
as desired.
\end{proof}

\begin{Rem}
A slightly different form of the previous result is:  
\begin{equation*}
\phi_{i\omega}(x) e^{t(1-x\coth(x))} = \int_{\mathbb{R}} \phi_{i\xi}(x) \frac{q_t(\xi-\omega) - q_t(\xi+\omega)}{2\omega \xi} \xi^2 d\xi. 
\end{equation*}
Since $\xi^2 d\xi$ is the Plancherel measure of the group $Sl(2,C)$ (see for instance \cite{Koor}, p. 8  and use the fact that $Sl(2,C)$ is isomorphic to $SO(3,1)$), Theorem \ref{Th2} gives the inverse spherical Fourier transform of $\phi_{\omega}\psi_t$ when $\omega$ belongs to the principal series (\cite{Koor}, p. 9). Moreover, since $q_t$ is even then the kernel 
\begin{equation*}
\frac{q_t(\xi-\omega) - q_t(\xi+\omega)}{2\omega \xi}
\end{equation*}
is so in both variables $(\omega, \xi)$. 
\end{Rem}

\begin{Rem}
The Markov property of the classical Markov process associated with the semi-group density $P_t(\omega, d\xi)$ may be deduced in this case directly from that of the L\'evy process. Indeed, it is equivalent to the Chapman-Kolmogorov equation: for any $\xi, \omega, \gamma, \in \mathbb{R}$,
\begin{equation*}
\gamma \int_{\mathbb{R}} \frac{q_t(\xi-\omega) - q_t(\xi+\omega)}{2\omega}\frac{q_s(\gamma-\xi) - q_s(\gamma + \xi)}{2\xi} \xi d\xi = \gamma \frac{q_{t+s}(\gamma-\omega) - q_{t+s}(\gamma+\omega)}{2\omega}.
\end{equation*}
\end{Rem}

\begin{Rem}
The fact that $P_t(\omega, d\xi)$ is a probability measure is equivalent to the harmonicity of the identity function with respect to $Q_t$: 
\begin{equation*}
\omega =  \int_{\mathbb{R}}\frac{q_t(\xi-\omega) - q_t(\xi+\omega)}{2} \xi d\xi = \int_{\mathbb{R}} q_t(\omega-\xi) \xi d\xi.
 \end{equation*}
In this respect, the action of the semi-group $(P_t)_{t \geq 0}$ on even functions in $C_0(\Omega)$ may be written as:  
\begin{equation*}
P_t(f)(\omega) = \int_{\mathbb{R}} f(\xi)\frac{\xi}{\omega} q_t(\xi-\omega)d\xi,
\end{equation*}
which allows to think of the corresponding Markov process as a Doob-tranfsorm of the L\'evy process. Furthermore, we can see by direct computations that the intertwining operator is given by the kernel: 
\begin{align*}
(\omega, x) \mapsto \int_{\mathbb{R}} \phi_{i\omega}(u) e^{iux} du &= \frac{\pi}{\omega}  \frac{\sinh(\pi \omega)}{\cosh(\pi \omega) + \cosh(\pi x)} 
\end{align*}
where the equality follows from formula 3.986. 2, in \cite{Gra-Ryz}. 
\end{Rem}

\subsection{The complementary series}
Now, let $\omega \in (-1,1)$ be in the complementary series and assume $t \geq |\omega|$. Then, $\phi_{\omega}\psi_t$ is still square integrable therefore we have a similar decomposition as in the previous theorem: 
\begin{equation*}
\phi_{\omega}(x)\psi_t(x) = \int_{\mathbb{R}} \frac{\sin(\xi x)}{\xi\sinh(x)} P_t(\omega, d\xi),
\end{equation*}
or equivalently 
\begin{equation*}
\frac{\sinh(\omega x)}{\omega} \psi_t(x) = \int_{\mathbb{R}} \frac{\sin(\xi x)}{\xi} P_t(\omega, d\xi).
\end{equation*}
More precisely, 
\begin{teo}\label{Th3}
Let $\omega \in (-1,1)$ and $t \geq |w|$. Then 
\begin{equation*}
\phi_{\omega}(x)\psi_t(x)  = \frac{e^t}{2\pi} \int_{\mathbb{R}} \phi_{i\xi}(x)  \frac{I_t^{-} (\omega, \xi) - I_t^{+} (\omega, \xi)}{\omega\xi} \xi^2 d\xi,
\end{equation*}
where $I_t^{\mp} (\omega, \xi)$ are absolutely convergent double series defined below.  
\end{teo}

\begin{proof}
Since $\phi_{\omega}(x) = \phi_{\omega}(-x)$ and from the uniqueness of the integral representation \eqref{Godement}, we deduce that the probability measure $P(\cdot, d\xi)$ is symmetric. It is also absolutely continuous with respect to Lebesgue measure on $\mathbb{R}$: indeed,  
\begin{equation*}
\partial_x  \left(\frac{\sinh(\omega\cdot)}{\omega}\psi_t\right)(x) = \int_{\mathbb{R}}e^{i\xi x}P_t(\omega, d\xi)
\end{equation*}
and the LHS of this equality is integrable as a function of the variable $x$. Hence, for any $\omega \in [0,1]$ and any $t \geq \omega$,
\begin{align*}
P_t(\omega,\xi) & = \frac{1}{2\pi} \int_{\mathbb{R}} e^{-ix\xi} \partial_x \left(\frac{\sinh(\omega\cdot)}{\omega}\psi_t\right)(x)dx
\\& = i \frac{\xi}{2\pi \omega} \int_{\mathbb{R}} e^{-ix\xi} \sinh(\omega x) \psi_t(x) dx
\\& = \xi \frac{e^t}{2\pi \omega} \int_{0}^{\infty} \sin(x\xi) (e^{-(t-\omega)x} - e^{-(t+\omega)x})\sum_{j \geq 0} L_j^{(-1)}(2tx)e^{-2jx}dx.
\end{align*}
Using the bound \eqref{bound}, we can split the above integral into the difference of the following series: 
\begin{equation*}
I_t^{\mp} (\omega, \xi) :=  \sum_{j \geq 0} \int_{0}^{\infty} \sin(x\xi) L_j^{(-1)}(2tx)e^{-(2j+t\mp \omega)x}dx, 
\end{equation*}
which, by the virtue of \eqref{Laguerre} together with formula 3.944. 5., in \cite{Gra-Ryz}, may be expanded as 
\begin{align*}
I_t^{\mp} (\omega, \xi) & = \sum_{j \geq 0}\sum_{m=0}^j  (m)_{j-m}\frac{(-2t)^m }{(j-m)!m!} \int_{\mathbb{R}} x^m \sin(x\xi) e^{-(2j+t\mp \omega)x} dx
\\& = \sum_{j \geq 0}\sum_{m=0}^j  (m)_{j-m}\frac{(-2t)^m }{(j-m)![(2j+t\mp \omega)^2+\xi^2]^{(m+1)/2}}\sin\left\{(m+1)\arctan\left(\frac{\xi}{2j+t\mp \omega}\right)\right\}.
\end{align*}
Moreover, the same arguments used in the proof of Theorem \ref{Th1} show that we can intertwine the summation order in such a way that 
\begin{align*}
I_t^{\mp} (\omega, \xi) & = \sum_{m\geq 0}\sum_{j \geq 0}  (m)_{j}\frac{(-2t)^m }{j![(2j+2m+t\mp \omega)^2+\xi^2]^{(m+1)/2}}\sin\left\{(m+1)\arctan\left(\frac{\xi}{2j+2m+t\mp \omega}\right)\right\}.
\end{align*}
Since $\phi_{\omega} = \phi_{-\omega}$ and $I_t^{-} (\omega, \xi) = I_t^{+} (-\omega, \xi)$, then the proof is complete.
\end{proof}

Finally, we deal with the case $t < |\omega|, \omega \in (-1,1) \setminus \{0\}$. Then, the spherical function $\phi_{\omega}$ decays as $e^{(|w|-t-1)x}, x \rightarrow +\infty$ so that a residual part shows up in the decomposition \eqref{Godement}. 
Actually, straightforward computations show that the map
\begin{equation*}
x \mapsto \phi_{\omega}(x)\psi_t(x) -  e^t\frac{|\omega|-t}{|\omega|}\phi_{|\omega|-t}(x) = \frac{e^t}{|\omega|\sinh(x)} [\sinh(|\omega|x)e^{-tx\coth(x)} - \sinh((|\omega|-t)x)]
\end{equation*}
is square integrable with respect to the radial part of the Haar measure\footnote{There is missed factor $e^t$ in \cite{Bia1}.}. As a matter of fact, it may be decomposed as: 
\begin{equation}\label{Dec1}
\phi_{\omega}(x)\psi_t(x) - e^t\frac{|\omega|-t}{|\omega|}\phi_{|\omega|-t}(x) = \int_{\mathbb{R}} \phi_{i\xi}(x)G_t(\omega, d\xi),
\end{equation}
for some (submarmovian) kernel $G_t(\omega, d\xi)$ and the latter may be computed along similar, yet more complicated, lines as those written in the proof of theorem \ref{Th3}. For sake of simplicity, we shall assume (without loss of generality) that $\omega \in (0,1)$ and prove that: 
\begin{teo}
The kernel of $G_t(\omega, d\xi)$ is absolutely continuous with respect to Lebsegue measure: 
\begin{equation*}
G_t(\omega, d\xi) = \xi \frac{e^t}{2\pi \omega} \left[J_t(\omega, \xi) - I_t^+(\omega, \xi) + \frac{1}{\sqrt{\xi^2+(\omega-t)^2)}}\sin\left(\arctan\left(\frac{\xi}{\omega-t}\right)\right)\right] d\xi,
\end{equation*}
where $J_t(\omega, \xi)$ is is displayed in \eqref{DoubleSer} below. 
\end{teo}

\begin{proof}
We start by writing \eqref{Dec1} as 
\begin{equation*}
\frac{e^t}{\omega} \partial_x \left[\sinh(\omega x)e^{-tx\coth(x)} - \sinh((\omega-t)x)\right] = \int_{\mathbb{R}} e^{ix\xi} G_t(\omega, d\xi),
\end{equation*}
and noting that the derivative of the function 
\begin{equation*}
x \mapsto \sinh(\omega x)e^{-tx\coth(x)} - \sinh((\omega-t)x)
\end{equation*}
is integrable. Consequently, 
\begin{align*}
G_t(\omega, d\xi) & = \frac{e^t}{2\pi \omega} \int_{\mathbb{R}} e^{-ix\xi} \partial_x \left[\sinh(\omega x)e^{-tx\coth(x)} - \sinh((\omega-t)x)\right] dx 
\\& = i\xi \frac{e^t}{2\pi \omega} \int_{\mathbb{R}} e^{-ix\xi} \left[\sinh(\omega x)e^{-tx\coth(x)} - \sinh((\omega-t)x)\right] dx
\\& = \xi \frac{e^t}{\pi \omega} \int_{\mathbb{R}} \sin(x\xi) \left[\sinh(\omega x)e^{-tx\coth(x)} - \sinh((\omega-t)x)\right] dx. 
\end{align*}
Now, split
\begin{align*}
\sinh(\omega x)e^{-tx\coth(x))} - \sinh((\omega-t)x) = \frac{e^{\omega x}}{2}[e^{-tx\coth(x)} - e^{-tx}] - \frac{1}{2}[e^{-(\omega+t\coth(x))x} - e^{-(\omega-t)x}], 
\end{align*}
and use the generating function \eqref{GenFun} to write: 
\begin{equation*}
e^{-tx\coth(x)} - e^{-tx} = \sum_{j \geq 1} L_j^{(-1)}(2tx)e^{-(2j+t)x}. 
\end{equation*} 
Together with the estimate \eqref{bound} lead to the following absolutely-convergent integrals: 
\begin{align*}
J_t(\omega, \xi) & := \int_0^{\infty}\sin(x\xi)\left\{\sum_{j \geq 1} L_j^{(-1)}(2tx)e^{-(2j+t-\omega)x}\right\} dx
\\& = \sum_{j \geq 1}\int_0^{\infty}\sin(x\xi)L_j^{(-1)}(2tx)e^{-(2j+t-\omega)x} dx 
\\& = \sum_{j\geq 1} \sum_{m=0}^j  (m)_{j-m}\frac{(-2t)^m }{(j-m)!m!}\int_0^{\infty}x^m \sin(x\xi)e^{-(2j+t-\omega)x} dx,
\end{align*}
and 
\begin{align*}
K_t(\omega, \xi) &:= \int_0^{\infty}\sin(x\xi) \left\{e^{-(\omega+t)x}\sum_{j \geq 0} L_j^{(-1)}(2tx)e^{-2jx}  - e^{-(\omega-t)x}\right\}dx
\\& =  \sum_{j \geq 0}\int_0^{\infty}\sin(x\xi)e^{-(\omega+t+2j)x} L_j^{(-1)}(2tx) dx -  \int_0^{\infty}\sin(x\xi)e^{-(\omega-t)x}dx
\\& = I_t^{+} (\omega, \xi) - \int_0^{\infty}\sin(x\xi)e^{-(\omega-t)x}dx.
\end{align*} 
Finally, formula 3.944. 5., in \cite{Gra-Ryz} provides the expressions: 
\begin{equation}\label{DoubleSer}
J_t(\omega, \xi) =  \sum_{j \geq 1}\sum_{m=0}^j  (m)_{j-m}\frac{(-2t)^m }{(j-m)![(2j+t- \omega)^2+\xi^2]^{(m+1)/2}}\sin\left\{(m+1)\arctan\left(\frac{\xi}{2j+t- \omega}\right)\right\}, 
\end{equation}
and 
\begin{equation*}
\int_0^{\infty}\sin(x\xi)e^{-(\omega-t)x}dx = \frac{1}{\sqrt{\xi^2+(\omega-t)^2)}}\sin\left(\arctan\left(\frac{\xi}{\omega-t}\right)\right), 
\end{equation*}
which finish the proof. 
\end{proof}

\section{$\omega=0$: the L\'evy stochastic area}
For $w=0$, Theorem \ref{Th2} specializes to 
\begin{align*}
\phi_0(x) \psi_t(x) = \frac{x}{\sinh(x)}\psi_t(x) & = \lim_{\omega \rightarrow 0} \int_{\mathbb{R}} \phi_{i\xi}(x) \frac{q_t(\xi-\omega) - q_t(\xi+\omega)}{2\omega \xi} \xi^2 d\xi
\\& = \frac{1}{\sinh(x)}\int_{\mathbb{R}} \sin(\xi x) \lim_{\omega \rightarrow 0}\frac{q_t(\xi-\omega) - q_t(\xi+\omega)}{2\omega} d\xi
\\& = \frac{1}{\sinh(x)}\int_{\mathbb{R}} \sin(\xi x) \partial_\xi q_t(\xi)d\xi.
\end{align*}
This result may be directly derived from an integration by parts, namely:  
\begin{equation*}
x \psi_t(x) = \int_{\mathbb{R}} x \cos(\xi x) q_t(\xi)d\xi = -\int_{\mathbb{R}} \sin(\xi x) \partial_\xi q_t(\xi) d\xi.
\end{equation*}
On the other hand, $\phi_0\psi_t$ is also the Euclidean Fourier transform of the L\'evy stochastic area at unit time conditional on the planar Brownian motion defining it being on the circle of radius $\sqrt{2t}$ (\cite{Gav}, Lemma 2). In \cite{Bia1}, the author asks for an explanation of this coincidence and we supply in the sequel some arguments supporting it. Of course, this is far from being a definitive answer. 

Our key observation is that the L\'evy stochastic area rather encodes the heat kernel of the so-called Landau Laplacian in the plane (the Schr\"odinger operator in the plane with a perpendicular constant magnetic field, \cite{Hel}): 
\begin{equation*}
H_b := -\frac{1}{2}\left[(\partial_x + ib y)^2 + (\partial_y - ibx)^2\right] 
\end{equation*}
where $b > 0$ is the strength of the magnetic field. This fact is not surprising and stems from the fact that this Hamiltonian may be obtained from the sublaplacian of the Heisenberg group after performing a partial Fourier transform with respect to the vertical variable. Now, $H_b$ may be realized as the image of a matrix $\mathcal{A}$ in the Lie algebra $sp(4,\mathbb{R})$ by the derived metaplectic representation $Mp(4,\mathbb{R})$ (see \cite{Mat-Uek} for further details). Indeed, let $\alpha:= b/(2\pi)$ and consider the matrix 
$\mathcal{A} \in sp(4,\mathbb{R})$: 
\begin{equation*}
\mathcal{A}_{\alpha}  = \left(\begin{array}{lr}
\alpha A & I_2 \\
-\alpha^2I_2 & \alpha A 
\end{array}\right), \qquad  A := \left(\begin{array}{lr}
0 & -1 \\
1 &  0
\end{array}\right),
\end{equation*}
where $I_2$ is the $2 \times 2$ identity matrix. Then, the spectrum of $H_b$ is the countable set $\{b(2m+1), m \geq 0\}$ and with $\pi$ denoting the metaplectic representation, we have
\begin{equation*}
H_b = \frac{d}{dt} (\pi [e^{-t\mathcal{A}_{\alpha}}])_{t=0} \equiv d\pi(\mathcal{A}_{\alpha})
\end{equation*} 
where $d\pi$ is the derived representation. Besides, its heat semi-group at unit time is given by (see e.g. \cite{Hel}, \cite{Mat-Uek}, eq. (6.1)): 
\begin{equation}\label{Subor}
\pi [e^{-\mathcal{A}_{\alpha}}](f)(0) = e^{-d\pi(A_{\alpha})} (f)(0) = \int_{\mathbb{R}^2}\phi_0(b) \psi_{|y|^2/2}(b) f(y)e^{-|y|^2/2} \frac{dy}{2\pi}.  
\end{equation}
for any square integrable function $f$. Thus, if $(\rho_y, v_y, \mathcal{H}_y), y \in \mathbb{R}^2,$ denotes the unitary representation of $Sl(2,\mathbb{C})$ associated with the positive definite function $\phi_0\psi_{|y|^2/2}$ via the GNS construction\footnote{$\mathcal{H}_y$ is the representation Hilbert space endowed with an inner product $\langle \cdot, \cdot \rangle_{\mathcal{H}_y}$ and $v_y \in \mathcal{H}_y $ is a cyclic $SU(2)$-invariant unit vector.}, then \eqref{Subor} may be written as an average of the continuous family of representations $(\rho_y, v_y, \mathcal{H}_y)_{y \in \mathbb{R}^2}$ with respect to the Gaussian distribution $e^{-|y|^2/2}/(2\pi)$:
\begin{equation*}
\pi [e^{-\mathcal{A}_{\alpha}}](f)(0) = \int_{\mathbb{R}^2} \langle \rho_y(\textrm{diag}(e^b, e^{-b})) v_y, v_y \rangle_{\mathcal{H}_y} e^{-|y|^2/2} \frac{dy}{2\pi}.
\end{equation*}
In this respect, it seems like the representations $(\rho_y)_{y \in \mathbb{R}^2}$ occur in the decomposition of $\pi$ into irreducible ones. 

This possible guess is indeed strengthened by the fact that the Cartan decomposition of $e^{-\mathcal{A}_{\alpha}}$ is somehow special in the sense that it may be turned into a Cartan decomposition in $Sl(2,\mathbb{C})$. 
More precisely, straightforward computations show that $\mathcal{A}_{\alpha}^2 = 2\alpha \mathcal{A}\mathcal{J}_{\alpha}$ where
\begin{equation*}
\mathcal{J}_{\alpha} := \left(\begin{array}{lr}
0_2 & (1/\alpha)I_2 \\
-\alpha I_2 & 0_2 
\end{array}\right) 
\end{equation*}
 satisfies $\mathcal{J}_{\alpha}^2 = -I_4$. As a result, one gets for any $j \geq 1$, 
\begin{equation*}
\mathcal{A}_{\alpha}^{2j} = (-1)^{j-1} (2\alpha)^{2j-1} \mathcal{A}_{\alpha} \mathcal{J}_{\alpha},  \quad \mathcal{A}_{\alpha}^{2j-1} = (-1)^{j-1} (2\alpha)^{2j-2} \mathcal{A}_{\alpha},
\end{equation*}
whence it follows that for any $t \geq 0$,
\begin{align*}
e^{-t\mathcal{A}_{\alpha}} & = I_4 + \frac{1-\cos (2\alpha t)}{2\alpha} \mathcal{A}_{\alpha}\mathcal{J}_{\alpha} + \frac{\sin(2\alpha t)}{2\alpha} \mathcal{A}_{\alpha} 
\\& = I_4 + \frac{\sin^2(\alpha t)}{\alpha} \mathcal{A}_{\alpha}\mathcal{J}_{\alpha} +\frac{\sin(\alpha t)\cos(\alpha t)}{\alpha}\mathcal{A}_{\alpha}. 
\end{align*}
or in block form:
\begin{equation*} 
e^{-t\mathcal{A}_{\alpha}}  = \left(\begin{array}{lr}
\cos^2(\alpha t) I_2 + \sin(\alpha t)\cos(\alpha t)A & [\sin^2(\alpha t) A+ \sin(\alpha t)\cos(\alpha t) I_2]/\alpha  \\
- \alpha [\sin(\alpha t)\cos(\alpha t) I_2 +\sin^2(\alpha t) A]  & \cos^2(\alpha t) I_2 + \sin(\alpha t)\cos(\alpha t)A
\end{array}\right). 
\end{equation*}
Besides, setting
\begin{eqnarray*}
U_{\alpha}(t)  &:= & \cos^2(\alpha t) I_2 + \sin(\alpha t)\cos(\alpha t)A \\ 
V_{\alpha}(t) &:= & \sin^2(\alpha t) A+ \sin(\alpha t)\cos(\alpha t) I_2,
\end{eqnarray*}
and noting that $A^2 = -I_2$, we readily get the relations
\begin{eqnarray*}
U_{\alpha}(t)U_{\alpha}^T(t) & = & \cos^2(\alpha t) I_2 \\ 
V_{\alpha}(t) V_{\alpha}^T(t) & = & \sin^2(\alpha t) I_2,  \\
U_{\alpha}(t)V_{\alpha}(t)^T &= & \sin(\alpha t)\cos(\alpha t) I_2 = U_{\alpha}(t)^TV_{\alpha}(t),
\end{eqnarray*}
which in turn show that:
\begin{equation*}
e^{-t\mathcal{A}_{\alpha}}e^{-t\mathcal{A}_{\alpha}^T} = \left(\begin{array}{lr}
\cos^2(\alpha t) + \displaystyle \frac{\sin^2(\alpha t)}{\alpha^2}  & \sin(\alpha t)\cos(\alpha t) \displaystyle \left(\frac{1}{\alpha} - \alpha\right)  \\
\sin(\alpha t)\cos(\alpha t) \displaystyle \left(\frac{1}{\alpha} - \alpha\right)  &  \cos^2(\alpha t) + \alpha^2 \sin^2(\alpha t)
\end{array}\right)  \quad \otimes \quad I_2,
\end{equation*}
and similarly 
\begin{equation*}
e^{-t\mathcal{A}_{\alpha}^T}e^{-t\mathcal{A}_{\alpha}} = \left(\begin{array}{lr}
\cos^2(\alpha t) + \alpha^2 \sin^2(\alpha t) & \sin(\alpha t)\cos(\alpha t) \displaystyle \left(\frac{1}{\alpha} - \alpha\right)  \\
\sin(\alpha t)\cos(\alpha t) \displaystyle \left(\frac{1}{\alpha} - \alpha\right)  & \cos^2(\alpha t) + \displaystyle \frac{\sin^2(\alpha t)}{\alpha^2}  
\end{array}\right)  \quad \otimes \quad I_2.
\end{equation*}
Consequently, the Cartan decomposition of $e^{-\mathcal{A}_{\alpha}}$ admits the block form: 
\begin{equation*}
e^{-\mathcal{A}_{\alpha}} = M_{\alpha}
\left(\begin{array}{lr} \lambda_1 I_2  & \\ & \lambda_2I_2 \end{array}\right) R_{\alpha}, 
\end{equation*}
where $M_{\alpha}, R_{\alpha} \in O(4) \cap Sp(4,\mathbb{R})$ are orthogonal symplectic matrices and $\lambda_1, \lambda_2 > 0, \lambda_1 \lambda_2 = 1$. Since the subgroup $O(4) \cap Sp(4,\mathbb{R}) \subset SO(4)$ consists of $2 \times 2$ block matrices 
\begin{equation*}
\left(\begin{array}{lr} 
G  & F \\ -F & G 
\end{array}\right),
\end{equation*}
then it may be identified with $U(2)$ through the map $(G,F) \mapsto G + iF$. As a matter of fact, we may assign to $e^{-\mathcal{A}_{\alpha}} $ the following matrix in $SL(2, \mathbb{C})$:  
\begin{equation*}
M_{\alpha}\left(\begin{array}{lr} \lambda_1  & \\ & 1/\lambda_1 \end{array}\right) R_{\alpha}. 
\end{equation*}

{\bf Acknowledgments}: we would like to thank Jacques Faraut, Philippe Biane, Bachir Bekka and Ren\'e Schilling for stimulating discussions and remarks.

\end{document}